\documentclass[letter,12pt]{article}
\usepackage{amssymb}
\usepackage{amsthm}
\usepackage{amsmath}
\usepackage{mathrsfs}
\usepackage{verbatim}
\usepackage[numbers]{natbib}
\usepackage{enumerate}
\usepackage{eucal}

\newcommand{\R}{\mathbb{R}}
\usepackage[colorlinks=true,pagebackref=false]{hyperref}
\hypersetup{urlcolor=blue, citecolor=red, linkcolor=blue}
\usepackage{color}
\newtheorem{theorem}{Theorem}
\newtheorem{corollary}[theorem]{Corollary}
\newtheorem{lemma}[theorem]{Lemma}
\newtheorem{proposition}[theorem]{Proposition}
\newtheorem{remark}[theorem]{Remark}
\newtheorem{example}[theorem]{Example}
\newtheorem{definition}[theorem]{Definition}
\newcommand{\ie}{i.e.,\,}
\newcommand{\eg}{e.g.,\,}
\newcommand{\Rsp}{\mathbb{R}}
\newcommand{\Prob}{\mathcal{P}}
\newcommand{\Class}{\mathcal{C}}

\newcommand{\one}{\mathbf{1}}
\newcommand{\dd}{\mathrm{d}}
\newcommand{\id}{\mathrm{id}}

\newcommand{\Wass}{\mathrm{W}}
\newcommand{\LL}{\mathrm{L}} 
\newcommand{\abs}[1]{\left|#1\right|}

\newcommand{\nr}[1]{\|#1\|}

\newcommand{\Leb}{\mathrm{Leb}}
\renewcommand{\bar}{\overline}

	\title{Robust risk management via multi-marginal optimal transport \footnote{The work of BP was partially supported by  National Sciences and Engineering Research Council of Canada Discovery Grant number  04658-2018.  He is also pleased to acknowledge the generous  hospitality of the Institut de Mathématiques d'Orsay, Université Paris-Sud during his stay in November of 2019 as a missionaire scientifique invité, when this work was partially completed. L.N. was supported by a public grant as part of the Investissement d'avenir project, reference ANR-11-LABX-0056-LMH, LabEx LMH,  from the CNRS PEPS JCJC (2022) and H-CODE.}}

	\author{Hamza Ennaji\thanks{Normandie Universit\'e, ENSICAEN, UNICAEN, CNRS, GREYC, France. \texttt{hamza.ennaji@unicaen.fr}}
	\and
	Quentin Mérigot\thanks{Université Paris-Saclay, CNRS, Laboratoire de mathématiques d’Orsay, 91405, Orsay, France.  \texttt{quentin.merigot@universite-paris-saclay.fr}}
	\and
                 Luca Nenna\thanks{Université Paris-Saclay, CNRS, Laboratoire de mathématiques d’Orsay, 91405, Orsay, France.  \texttt{luca.nenna@universite-paris-saclay.fr}}  
\and 
                 Brendan Pass\thanks{Department of Mathematical and Statistical Sciences, 632 CAB, University of Alberta, Edmonton, Alberta, Canada, T6G 2G1  \texttt{pass@ualberta.ca}} }

\begin{document}

%
	\maketitle
	
	\begin{abstract}
		We study the problem of maximizing a spectral risk measure of a given output function which depends on several underlying variables, whose individual distributions are known but whose joint distribution is not.  We establish and exploit an equivalence between this problem and a multi-marginal optimal transport problem.  We use this reformulation to establish explicit, closed form solutions when the underlying variables are one dimensional, for a large class of output functions.  For higher dimensional underlying variables, we identify conditions on the output function and marginal distributions under which solutions concentrate on graphs over the first variable and are unique, and, for general output functions, we find upper bounds on the dimension of the support of the solution.  We also establish a stability result on the maximal value and maximizing joint distributions when the output function, marginal distributions and spectral function are perturbed; in addition, when the variables one dimensional, we show that the optimal value exhibits Lipschitz dependence on the marginal distributions for a certain class of output functions.  Finally, we show that the equivalence to a multi-marginal optimal transport problem extends to maximal correlation measures of multi-dimensional risks; in this setting, we again establish conditions under which the solution concentrates on a graph over the first marginal.
	\end{abstract}
	\tableofcontents	
	
	\section{Introduction}
	In a variety of problems in operations research, a variable of interest $b=b(x_1,x_2,\dotsc,x_d)$ depends on 
	several underlying random variables, whose individual distributions are known (or can be estimated) but whose joint distribution is not.  A natural example arises in finance, when one considers the payout of a derivative depending on several underlying assets.  An estimate of the distribution of the asset values themselves can often be inferred from the prices of vanilla call and put options for a wide range of    
	strike prices; since these options are widely traded, their prices are readily available.  However, estimating the joint distribution would require prices of a wide range of derivatives with payouts depending on all the variables, which are typically much scarcer.  Other examples arise when the variables represent parameters in a physical system whose individual distributions can be estimated empirically or through modeling (or a combination of both) but whose dependence structure cannot; an example of this flavour, originating in \cite{iooss2015review}, in which the output $b$ represents the simulated height of a river at risk of flooding, and the underlying variables include various design parameters and climate dependent factors can be found in Section \ref{subsect: river example} below.  A third example comes from election projections in political science, where $b$ might represent the outcome of an election and the $x_i$ vote shares in different regions, whose individual distributions can be modelled from polling data. 
	
	Methods used in risk management to evaluate the resulting aggregate level of risk naturally depend on the distribution of the output variable $b$, and therefore, in turn, on the joint distribution of the $x_i$.  A natural problem is therefore to determine bounds, or worst-case scenarios, on these risk measures; that is, to maximize a given risk  measure over all possible joint distributions of the $x_i$ with known marginal distributions.  
	Problems such as this have received extensive attention within the financial risk management community; see, for example, \cite{Ruschendorf2013}.  In the simplest of these problems, the underlying variables $x_i$ are typically real valued and the output variable $b$ is often assumed to have a particular structure (in many cases, it is a weighted sum of the $x_i$, reflecting the value of a portfolio built out of underlying assets with values $x_i$, or a function of this weighted sum).  In these cases, explicit solutions for the maximizing couplings of the $x_i$ can sometimes be obtained (see, for example, \cite{Puccetti2013,PuccettiRuschendorf13,WangWang11}).

	For more general output functions $b$, and possibly multi-dimensional underlying variables $x_i$, much less is known about the dependence structure of the maximizing joint distributions. A recent paper of Ghossoub-Hall-Saunders \cite{ghossoub2020maximum} studies this more general setting systematically, allowing the underlying variables to take values in very general spaces and the output $b$ to take a very general form, and focuses on \emph{spectral risk measures} (see Definition \ref{def: spectral risk measure} below). They observe that the resulting problem is a generalization of optimal transport (see Section 2 for a brief overview of the optimal transport problem); in fact, in the simplest case, when the spectral function is identically equal to one and the number $d$ of underlying variables is $2$, the problem is exactly a classical optimal transport problem with surplus function $b$.   For more general spectral risk measures, the problem has the  \emph{same constraint set} as the optimal transport problem, but a more general, \emph{non-linear objective functional}; they adapt the duality theory of optimal transport to this setting and establish results on the stability with respect to perturbations of $b$ and the marginal distributions.  Although the analysis in \cite{ghossoub2020maximum} focused on the $d=2$ case, they note that their results can be extended to the $d \geq 3$ setting, in which case maximizing spectral risks becomes a generalization of the multi-marginal optimal transport problem (see Section \ref{section:prelim} for a brief overview of this problem).  As in the two marginal case, one maximizes a concave objective function over the set of couplings of the given marginals -- when the spectral function is constant, one obtains multi-marginal optimal transport, a linear maximization.
	
	The main purpose of this paper is to establish and exploit a simpler, but equivalent, reformulation of this problem. Specifically, we show that for \emph{any} spectral risk measure, the maximization can in fact be formulated as a traditional multi-marginal optimal transport problem with $d+1$ marginals: the given marginals distributions of the $x_i$ as well as another distribution arising from the particular form of the spectral function.  Although this formulation slightly increases the underlying dimension, it makes the maximization problem linear and much more tractable -- indeed, the results and techniques, both theoretical and computational, in the substantial  literature on multi-marginal optimal transport become directly applicable.

	Though the structure of solutions to multi-marginal optimal transport problems is in general a notoriously delicate issue (see \cite{PassSurvey} for an overview) there are many important cases when the problem is well understood.  As a consequence of our result, when the underlying variables $x_i$ are all one dimensional, we derive an explicit characterization of solutions for  a substantial class of output functions $b$, through a careful refinement of the existing theory of multi-marginal optimal transport on one dimensional ambient spaces.  For $b$ falling outside this class, explicit solutions are likely generally unattainable; however, our formulation of the problem can potentially facilitate the use of a very broad range of computational methods for optimal transport to approximate solutions numerically (see Section 8.1 in \cite{Benamou2021} and the references therein).
	
	For underlying variables in higher dimensional spaces, explicit solutions are generally not possible. However, we show that known techniques from multi-marginal optimal transport can be adapted to identify conditions under which the solution concentrates on a graph over the first variable; 
	solutions having this structure, commonly referred to as \emph{Monge solutions} in the optimal transport literature, are analogous to the well known comonotone couplings which often maximize spectral risk measures in the one dimensional case, since they essentially assert that knowing the value of the first variable completely determines the values of the others. On the other hand, for many output functions $b$, conditions ensuring this structure of a solution do not hold; in this case, we demonstrate that available techniques can still be used bound the dimension of the support of the optimizer. 	In addition to being of theoretical interest, this fact is potentially important in future work on the computation of solutions.\footnote{In general, the joint distributions of interest are probability measures on an $n\cdot d$ dimensional space, where $d$ is the number of variables, and $n$ the dimension of each variable (assuming, for simplicity of exposition, that all variables are of the same dimension).  When maximizers concentrate on graphs over $x_1$, they are essentially much simpler, $n$-dimensional objects.  More generally, it is  possible to prove that optimizers  concentrate on $m$ dimensional subsets of the $n\cdot d$ dimensional ambient space, where $m$ satisfies $n \leq m \leq n\cdot d$ and is determined by the output function $b$; this has consequences for the covering number of the support, which will potentially play an important role in future computation.} 
	
	We go on to show that our formulation, combined with standard stability results for optimal transport problems, implies that the  maximal value of the spectral risk measure, and the maximizing joint distributions, are stable with respect to perturbations of the marginals, output function and spectral function.  When the underlying variables are one dimensional and the output function satisfies additional hypotheses, we further leverage the connection to optimal transport to show that the maximal value's dependence on the marginals is in fact Lipschitz.
	
	We take particular note of the special but important case when the spectral risk measure in question is the \emph{expected shortfall}; in this case, the problem further reduces to a \emph{multi-marginal partial transport problem} on the $d$ original distributions. In this case, in higher dimensions, results in \cite{CaffarelliMcCann10,Figalli10} and \cite{KitPass} immediately apply, implying uniqueness and graphical structure of solutions for a certain class of output functions $b$.  Furthermore, the fact, exposed by our work here, that optimal partial transport can be reformulated as a multi-marginal optimal transport problem does not seem to have been observed before, and may be of independent interest in the optimal transport community.
	
	We also extend these ideas to the setting where the output function $b$ is multi-variate valued, using the maximal correlation risk measures developed in \cite{ekeland2012comonotonic}.  We are again able to identify conditions under which the solution concentrates on a graph. 
	
	The paper is organized as follows: in Section \ref{section:prelim}, we first introduce the traditional (two marginal) optimal transport problem and its multi-marginal generalization, and establish certain preliminary results which we will need. We then introduce spectral risk measures, and the problem of maximizing them over joint distributions with fixed marginals, which is the main object of interest in this paper (see \eqref{eqn: max alpha risk} below).  In Section \ref{section: equivalence}, we establish the equivalence between this maximization problem and multi-marginal optimal transport.  In Section \ref{section:1d}, we establish explicit solutions for the maximizing joint distributions when the underlying marginal variables are one dimensional and  in the fifth section, we study how these result generalize to the setting where the underlying variables are higher dimensional. Section \ref{section:stability} is focused on the stability of the problem with respect to variations in the marginals, output function and spectral risk measure, while the final section is reserved for the extension of some of our results to maximal correlation risk measures on multi-variate output functions.

	
	\section{Preliminaries and problem formulation}\label{section:prelim}
	
	\paragraph{Assumption and notations}
	Let $\mu_i \in \Prob(X_i)$ be probability measures on Polish spaces $X_i$, for $i=0,1,\dotsc,d$. Given a mapping $T:X_i \rightarrow X_j$, we say that $T$ pushes $\mu_i$ forward to $\mu_j$, and write 
	\[ T_\#\mu_i = \mu_j,\; \text{if}\; \mu_i(T^{-1}(A)) =\mu_j(A)\; \text{for all Borel}\; A \subseteq X_j.\]  We will let $\Gamma(\mu_0, \mu_1,\dotsc,\mu_d) \subset \Prob(X_0 \times X_1 \times\dotsm\times X_d)$ denote the set of probability measures on $X_0 \times X_1 \times \dotsm\times X_d$ whose marginals are the $\mu_i$; that is, $\mu_i=\Big((x_0,x_1,x_2,\dotsc,x_d) \mapsto x_i\Big)_\#\pi$ for each $i$ and each $\pi \in \Gamma(\mu_0, \mu_1,\dotsc,\mu_d)$.
	\begin{definition}[$(1,\dotsc,d)$-marginal]
	    For $\pi \in \Prob(X_0 \times X_1 \times \dotsm\times X_d)$, we call the \emph{$(1,\dotsc,d)$-marginal} of $\pi$, the projection of $\pi$ onto $X_1 \times\dotsm\times X_d$; that is, $\Big((x_0,x_1,x_2,\dotsc,x_d) \mapsto (x_1,x_2,\dotsc,x_d)\Big)_\#\pi$. Note that if $ \pi \in \Gamma(\mu_0, \mu_1,\dotsc,\mu_d)$ then clearly its $(1,\dotsc,d)$-marginal is in $\Gamma( \mu_1,\dotsc,\mu_d)$.
	\end{definition}
	
	As usual we define the cumulative distribution function (c.d.f) of a probability measure
	$\mu\in\Prob(\Rsp)$ as $F_\mu(x) = \mu((-\infty,x])$. The generalized or  pseudo-inverse of $F_\mu$ is defined as
	$$ F_\mu^{-1}(m) = \inf\{x\in\Rsp\;|\; F_\mu(x)\geq m \}.$$ 
	We recall that $F_\mu^{-1}(m)$ is the value of the $m$th quantile.
	\paragraph{Optimal Transport problem}
	Given two probability distributions $\mu_0$ and $\mu_1$ on Polish spaces $X_0$ and $X_1$, respectively,  and a surplus function $s: X_0\times X_1\to \R$, the optimal transport problem (in its Monge formulation) consists in maximising 
	\begin{equation}
		\label{eq:monge}
		\int_{X_0} s(x,T(x)) \dd\mu_0(x)
	\end{equation}
	under the constraint that $T_\# \mu_0=\mu_1$ (namely $\mu_1$ is the image measure of $\mu_0$ through the map $T$).
	This is a delicate problem since the mass conservation constraint $T_\# \mu_0=\mu_1$ is highly nonlinear (and the feasible set may even be empty, for instance if $\mu_0$ is a Dirac mass and $\mu_1$ is not). For this reason, it is common to study  a relaxed formulation of \eqref{eq:monge} which allows mass splitting; that is, 
	\begin{equation}
		\label{eq:kant}
		\max_{\pi\in\Gamma(\mu,\nu)}\int_{X_0\times X_1} s(x,y) \dd\pi(x,y)
	\end{equation}
	where, as above, $\Gamma(\mu_0,\mu_1)$ consists of all probability measures on $X_0\times X_1$ having $\mu_0$ and $\mu_1$ as marginals. 
	Note that this is a (infinite dimensional) linear programming problem and that there exists solutions under very mild assumptions on $s, \mu_0$ and $\mu_1$.
	A solution $\pi$ of \eqref{eq:kant} is called an optimal transport plan; in particular,   if an optimal plan of \eqref{eq:kant} has a deterministic form $\pi = (\id,T)_\sharp\mu_0$ (which means that no splitting of mass occurs and $\pi$ is concentrated on the graph of $T$) then $T$ is  an optimal transport map, \ie a solution to \eqref{eq:monge}.  It is therefore sometimes called a \emph{Monge solution} of \eqref{eq:kant}.
	The linear problem \eqref{eq:kant}  also has a convenient dual formulation
	\begin{equation}
		\label{eq:dual}
		\min_{\substack{u,v\\ u(x)+v(y)\geq s(x,y)}}\int_{X_0}u(x)\dd\mu_0(x)+\int_{X_1}v(y)\dd\mu_1(y)
	\end{equation}
	where $u(x)$ and $v(y)$ are the so-called Kantorovich potentials. Optimal Transport theory for two marginals has developed very rapidly in the 25 last years; there are well known conditions on $s$, $\mu_0$ and $\mu_1$ which guarantee that there is a unique deterministic optimal plan   and we refer to the textbooks by Santambrogio \cite{santambook}  and Villani \cite{Villani-TOT2003,Villani-OptimalTransport-09}, for a detailed exposition.

	
	\paragraph{Multi-marginal optimal transport}
	In this paper  we consider a generalization of \eqref{eq:monge} to the case in which more than two marginals are involved: 
	given probability measures $\mu_0,\mu_1,\dotsc,\mu_d$ on spaces $X_0,X_1,\dotsc,X_d$, and a surplus function $s(x_0,x_1,\dotsc,x_d)$, the multi-marginal optimal transport problem is to maximize
	\begin{equation}\label{eqn: multi-marginal ot}
		\int_{X_0\times X_1 \times\dotsm\times X_d} s(x_0,x_1,\dotsc,x_d)\dd\pi(x_0,x_1,\dotsc,x_d)
	\end{equation}
	over $\pi \in \Gamma(\mu_0,\mu_1,\dotsc,\mu_d)$.  
	This problem has attracted increasing attention in recent years since it arises naturally in many different settings such as economics \cite{Ekeland05}, quantum chemistry \cite{buttazzo2012optimal,cotar2013density}, fluid mechanics \cite{brenier1989least}, etc. Our results in Section \ref{section: equivalence} will bring forth a new application to maximizing spectral risk measures.
	
	It is well known that under mild conditions (for example, continuity of $s$ and compact support of the $\mu_i$ is more than sufficient) \eqref{eqn: multi-marginal ot} admits a solution; see, for example, \cite{PassSurvey}.
	Although the structure of solutions to \eqref{eqn: multi-marginal ot} can generally depend delicately on $s$, there is a growing theory, and an important class for which solutions can be derived essentially explicitly, which we describe below.  In addition, there are now powerful numerical tools to compute solutions for problems falling outside this class (for a more detailed discussion, we refer to Section 8.1 in \cite{Benamou2021} and the references therein).
	
	The most tractable of these problems occurs when the underlying asset spaces are one dimensional, and the mixed second derivatives of $s$ interact in a certain way.
	\begin{definition}[Compatibility]
		Suppose each $X_i \subset \mathbb{R}$ is a bounded real interval.  Assume that $s \in \Class^2(X_0 \times X_1 \times \dotsm\times X_d)$.  We say that $s$ is strictly compatible  if  for each three distinct indices $i,j,k \in \{0,1,\dotsc,d\}$ and each $(x_0,\ldots,x_d) \in X_0 \times \dotsm\times X_d$ we have
		\begin{equation}\label{eqn: compatibility}
			\frac{\partial ^2s}{\partial x_i\partial x_j}\Bigg[\frac{\partial ^2s}{\partial x_k\partial x_j}\Bigg]\frac{\partial ^2s}{\partial x_k\partial x_i}(x_0,\dotsc,x_d) >0.
		\end{equation}
		We will say that $s$ is weakly compatible (or simply compatible) if for each $i \neq j$, we have either $\frac{\partial ^2s}{\partial x_i\partial x_j} \geq 0$ throughout $X_0 \times X_1 \times \dotsm\times X_d$ or $\frac{\partial ^2s}{\partial x_i\partial x_j} \leq 0$ throughout $X_0 \times X_1 \times \dotsm\times X_d$, and 
		\begin{equation}\label{eqn: weak compatibility}
			\frac{\partial ^2s}{\partial x_i\partial x_j}\Bigg[\frac{\partial ^2s}{\partial x_k\partial x_j}\Bigg]\frac{\partial ^2s}{\partial x_k\partial x_i}(x_0,\dotsc,x_d) \geq 0,
		\end{equation}
		for each distinct $i,j,k \in \{0,1,\dotsc,d\}$ and each $(x_0,\ldots,x_d) \in X_0 \times \dotsm\times X_d$.
	\end{definition} 
	The fact that the mixed partials $\frac{\partial ^2s}{\partial x_i\partial x_j}$ do not change signs under the compatibility condition allows us to  partition the set $\{0,1,2,\dotsc,d\} =S_+\cup S_-$ of indices into disjoint subsets $S_+$ and $S_-$ such that $0\in S_+$ and for each $i \neq j$, $ \frac{\partial ^2s}{\partial x_i\partial x_j}\geq 0 $ throughout $X_0 \times X_1 \times \dotsm\times X_d$ if either both $i$ and $j$ are in $S_+$ or if both are in  $S_-$, and $ \frac{\partial ^2s}{\partial x_i\partial x_j}\leq 0 $ throughout $X_0 \times X_1 \times \dotsm\times X_d$ otherwise. The inequalities are strict if strict compatibility holds .
    A well known special case of compatibility is captured by the following definition.
	\begin{definition}[Supermodularity]
		Suppose each $X_i \subset \mathbb{R}$ is a connected real interval.
		Assume that $s \in \Class^2(X_0 \times X_1 \times \dotsm\times X_d)$.  We say that $s$ is  supermodular if $\frac{\partial ^2s}{\partial x_i\partial x_j} \geq0$ for all $i \neq j$; we say $s$ is strictly supermodular if $\frac{\partial ^2s}{\partial x_i\partial x_j} >0$ for all $i \neq j$
	\end{definition}

	\begin{remark}\label{rem: supermodularity vs compatibility}
		Notice that if  $s$ is  supermodular, then it is clearly compatible.  On the other hand, if $s$ is compatible, define $\tilde s(\tilde x_0,\tilde x_1,\dotsc,x_d) =s(x_0,x_1,\dotsc,x_d)$, where $\tilde x_0=x_0$ and for each $i=1,2,\dotsc,d$ we set  $\tilde x_i = x_i$ if $i \in S_+$ and $\tilde x_i = -x_i$ if $i \in S_-$.  Then it is straightforward to show that $\tilde s$ is supermodular.  Therefore, compatibility is simply  supermodularity up to a change of variables.
		
		Both supermodularity and compatibility may be extended to non $C^2$ functions, by considering partial differences; for simplicity of exposition, we stick with the version relying on second derivatives here.
		
	\end{remark}
	We then introduce the following special coupling in $\Gamma(\mu_0,\mu_1,\dotsc,\mu_d)$.

	\begin{definition}
		For a compatible $s$, we define the \emph{$s$ - comonotone coupling} by:
		\begin{equation}\label{eqn: monotone solution}
			\pi =(G_0,G_1,\dotsc,G_d)_\# \Leb_{[0,1]},
		\end{equation} where $G_0 = F^{-1}_{\mu_0}$ and for each $i =1,2,\ldots,d$ 
		\begin{equation}
			G_i(m)=\begin{cases}
				F^{-1}_{\mu_i}(m)\; &\hbox{if}\; i \in S_+,\\
				F^{-1}_{\mu_i}(1-m)\; &\hbox{if}\; i \in S_-.
			\end{cases}
		\end{equation}\footnote{There may be multiple ways to define the $G_i$ (if, for example, for some $i$, $ \frac{\partial ^2s}{\partial x_i\partial x_j}=0 $ everywhere for all $j$); in this case, the results below apply to all resulting such $s$ - comonotone couplings.  Note, however, that the $G_i$ are uniquely defined when compatibility is strict. }
	\end{definition}
	\begin{remark}
		If $s$ is supermodular, we can take $S_-$ to be empty and  $S_+ =\{0,1,\dotsc,n\}$; the $s$-comonotone coupling is then exactly the well known classical comonotone coupling. 
	\end{remark}
	The following result can be found in \cite{carlier2003class} for strictly supermodular functions as well as in \cite{Passthesis} for compatible functions (where its proof appears together with the formulation of the compatibility conditions and the observation that it is equivalent to strict supermodularity after changing coordinates).
	
	\begin{theorem}
		\label{thm:monge state}
		Suppose that $s$ is strictly compatible.  Then the $s$-comonotone coupling \eqref{eqn: monotone solution} is the unique optimizer in \eqref{eqn: multi-marginal ot}.
		
	\end{theorem}

	\begin{remark}\label{rem: 2 d monotone}
		A classical but important case occurs when $d=2$ and $s(x_0,x_1) =x_0x_1$.  It is well known that there is only one measure in $\Gamma(\mu_0,\mu_1)$ with monotone increasing support.  This will be used in the proof of Lemma \ref{lem: monotone 2d solutions} below.  
		
		For $d \geq 3$, the $s$-comonotone $\pi$ is characterized by the fact that its twofold marginals, $\pi_{ij} =((x_0,x_1,\ldots,x_d) \mapsto (x_i,x_j))_\# \pi$ are the monotone increasing (respectively decreasing) couplings if $\frac{\partial ^2s}{\partial x_i\partial x_j} >0$ (respectively $\frac{\partial ^2s}{\partial x_i\partial x_j} <0$); compatibility ensures existence of a $\pi$ with this property.
	\end{remark}
	Though the following result does not seem to be available in the literature, it is a straightforward consequence of Theorem \ref{thm:monge state} and the well known stability of optimal transport with respect to perturbations of the cost function.  It asserts that the optimality of the measure constructed in Theorem \ref{thm:monge state} still holds if the strong compatibility assumption is relaxed to weak compatibility, although the uniqueness assertion may not.
	\begin{proposition}\label{prop: non unique monge state}
		Suppose that $s$ is weakly compatible.  Then the $s$-comonotone coupling \eqref{eqn: monotone solution} is optimal in \eqref{eqn: multi-marginal ot}.
	\end{proposition}
	
	\begin{proof}
		For $\epsilon >0$, define
		\[\begin{split}
		s_\epsilon(x_0,x_1,\dotsc,x_d) &=s(x_0,x_1,\dotsc,x_d) +\epsilon \sum_{i,j \in S_+ }(x_i-x_j)^2 +\epsilon \sum_{i,j \in S_- }(x_i-x_j)^2\\ &-\epsilon \sum_{i \in S_+,j \in S_- }(x_i-x_j)^2		    
		\end{split}\]
		notice now that $s_\epsilon$ is strictly compatible and so Theorem \ref{thm:monge state} implies that $\pi =(G_0,G_1,\dotsc,G_d)_\# \Leb_{[0,1]}$ is optimal in \eqref{eqn: multi-marginal ot} for each $\epsilon$.  Therefore, for any other $\tilde \pi \in \Gamma(\mu_0,\mu_1,\dotsc,\mu_d)$ we have
		$$
		\int s_\epsilon \dd\tilde\pi \leq \int s_\epsilon \dd \pi
		$$
		
		Since the $s_\epsilon$ converge uniformly to $s$ as $\epsilon \rightarrow 0$, we can take the limit in the inequality above and obtain
		$$
		\int s \dd\tilde\pi \leq \int s \dd \pi.
		$$
		As  $\tilde \pi \in \Gamma(\mu_0,\mu_1,\dotsc,\mu_d)$ was arbitrary, this yields the desired result.
	\end{proof}
	

	We also note that, as for the two marginal problem,  \eqref{eqn: multi-marginal ot} admits a very useful dual formulation:
	
	\begin{equation}
		\label{eqn:dualit-multi-ot}
		\begin{split}
		\inf\bigg\{\sum_{i=0}^{d}\int_{X_i}u_i(x_i)\dd\mu_i(x_i)\;&|\;u_i\in\Class_b(X_i),\forall i=0,\dotsc,d,\\ &\sum_{i=0}^{d} u_i(x_i)\geq s(x_0,\dotsc,x_d)\bigg\}.		    
		\end{split}
	\end{equation}
	As with the primal problem \eqref{eqn: multi-marginal ot}, it is well known that under mild conditions a solution $(u_0,u_1,\dotsc,u_d)$ to \eqref{eqn:dualit-multi-ot} exists; see, for example, \cite{gangbo1998optimal,CarlierNazaret08,Pass11}.
	In particular, given optimal solutions $\pi$ and $(u_0,\dotsc,u_d)$ to \eqref{eqn: multi-marginal ot} and \eqref{eqn:dualit-multi-ot}, respectively,  the following optimality condition holds
	\begin{equation}
		\label{eqn: opt condition}
		\sum_{i=0}^{d} u_i(x_i) = s(x_0,\dotsc,x_d),\; \pi-\text{a.e.}
	\end{equation}
	moreover, by Proposition \ref{prop: non unique monge state}  for a compatible $s$  \eqref{eqn: opt condition} rewrites
	\[ \sum_{i=0}^{d} u_i(G_i(m)) = s(G_0(m),\dotsc,G_d(m)),\; \Leb-\text{a.e..} \]

	
	\subsection{Maximal spectral risk measures over couplings of given marginals}
	
	Let the real line describe a certain level of risk  (\eg the level of radiation in the nuclear power plant) and $\mu$ be a probability measure on $\Rsp$ which can be interpreted as the distribution of risk. We will consider the following form of quantifier of the risk associated with $\mu$.
	\begin{definition}[spectral risk measure]\label{def: spectral risk measure}
		A functional $\mathcal{R}:\mathcal{P}(\mathbb{R}) \rightarrow \mathbb{R}$ is called a \emph{spectral risk measure} if it takes the form $\mathcal{R}=\mathcal{R}_\alpha$, where	
		\begin{equation}\label{eqn: alpha risk}
			\mathcal{R}_\alpha(\mu):=	\int_0^1F_\mu^{-1}(m)\alpha(m)\dd m.
		\end{equation}
	\end{definition}
	\begin{remark}
	Note that by taking $\alpha$ to be defined and real valued on $[0,1]$, we are tacitly assuming that it is bounded, since $\alpha(0) \leq \alpha(m) \leq \alpha(1)$ by monotonicity; the same assumption was made in \cite{ghossoub2020maximum}.  This is mostly for technical convenience; we expect that most of our results can be extended to the more general case $\alpha:[0,1) \rightarrow \mathbb{R}_+$, where possibly $\lim_{m \rightarrow 1}\alpha(m) =\infty$, under suitable hypothesis (for instance, decay conditions on $\alpha_\#\Leb_{[0,1]}$).
	\end{remark}
	For a given $\alpha$, we will often refer to the spectral risk measure $\mathcal{R}_\alpha$ as the \emph{$\alpha$-risk}, and $\alpha$ as the the \emph{spectral function.}
	
	Spectral risk measures are pervasive in risk management and insurance.  Indeed, it is well known that any coherent risk measure which is additive for comontonic random variables must be spectral (see for instance \cite{FollmerSchied2002}).

	We begin with the following  variational characterization of spectral risk measures.  This result is well known to experts (see, for example, \cite{Ruschendorf2006} and \cite{ekeland2012comonotonic}, in which this fact is the basis for a multi-variate extension).  We include a brief proof here for the convenience of the reader
	
	
	\begin{lemma}\label{lem: ot formula for R} 
		For any spectral risk measure and any probability measure $\mu$ with $\int_{\mathbb{R}} x \dd\mu(x) > -\infty$, 
		\begin{equation}\label{eq:R-alpha-variational}
			\mathcal{R}_\alpha(\mu) = \max_{\pi\in \Gamma(\alpha_{\sharp}\Leb_{[0,1]},\mu)} \int_{\Rsp\times \Rsp} xy \dd \pi(x,y),
		\end{equation}
		where we have denoted $\Gamma(\alpha_{\sharp}\Leb_{[0,1]},\mu)$ the space of probability
		measures on $\Rsp^2$ with marginals $\alpha_{\sharp}\Leb_{[0,1]}$ and $\mu$.
	\end{lemma}
	\begin{proof}
		Let us re-write the right hand side of \eqref{eq:R-alpha-variational} as the following optimal transport problem between $\mu_0=\alpha_{\sharp}\Leb_{[0,1]}$ and $\mu_1=\mu$
		\[ \max_{\gamma\in \Gamma(\mu_0,\mu_1)} \int_{[0,1]\times \Rsp}s(x,y) \dd \pi(x,y), \]
		where $s(x,y)=xy$.
		Then, since the surplus is strictly supermodular there exists a unique optimal $\pi$ of the form $\pi=(G_0,G_1)_\sharp\Leb_{[0,1]}$, where $G_0(m)=F^{-1}_{\mu_0}(m)=F^{-1}_{\alpha_\sharp\Leb_{[0,1]}}(m)=\alpha(m)$ and $G_1(m)=F^{-1}_{\mu}(m)$, by Theorem \ref{thm:monge state}. The desired result follows.
	\end{proof}
	
	
	\begin{corollary} 
		\label{coro: concavity}
		$\mathcal{R}_\alpha$ is concave on $\Prob(\Rsp)$.
	\end{corollary}
	\begin{proof}
		Let $\mu_0,\mu_1\in \Prob(\Rsp)$ and let $\pi_0,\pi_1$ be
		optimal for the problem \eqref{eq:R-alpha-variational}. Then, with
		$\mu_t = (1-t)\mu_0 + t \mu_1$ and $\pi_t = (1-t)\gamma_0+\gamma_1$
		one has $\pi_t \in \Gamma(\alpha_\sharp\Leb_{[0,1]},\mu_t)$ so that
		\[
		\begin{split}
			\mathcal{R}_\alpha(\mu_t)
			&\geq \int x y \dd \pi_t(x,y) \\
			&= (1-t) \int x y \dd \pi_0(x,y) + t  \int x y \dd \pi_1(x,y) \\
			&= (1-t) \mathcal{R}_\alpha(\mu_0) + t \mathcal{R}_\alpha(\mu_1).
		\end{split}
		\]
	\end{proof}
	
	Return now to the case in which there are  $d$ parameters which enter into the estimation of risk through the output function $b$, and we wish to evaluate the worst case scenario for the $\alpha$-risk $\mathcal{R}_\alpha(b_\#\gamma)$ of $b_\#\gamma$ among couplings $\gamma \in \Gamma(\mu_1, \mu_2,\dotsc,\mu_d)$ of the marginal distributions $\mu_i$.  That is, we want to maximize:
	\begin{equation}\label{eqn: max alpha risk}
		\max_{\gamma \in \Gamma(\mu_1, \mu_2,\dotsc,\mu_d)} \mathcal{R}_\alpha(b_\#\gamma).
	\end{equation}
	\begin{remark}[Maximizing Expected Shortfall is optimal partial transport]\label{rem:: expected shortfall -partial transport}
		A special case of particular importance in risk management applications occurs when $\alpha = \alpha_{m_0}: = \frac{1}{m_0} \one_{[1-m_0,1]}$, in which case  \eqref{eqn: alpha risk} is also know as the \emph{Expected Shortfall}, or \emph{Conditional Value at Risk}. 
		
		In this setting,  the maximization problem \eqref{eqn: max alpha risk} is actually equivalent to a well known variant of the optimal transport problem; indeed, it can be reformulated into
		\[ \max_{\gamma \in \Gamma_{m_0}(\mu_1,\dotsc,\mu_d)} \frac{1}{m_0} \int
		b(x_1,\dotsc, x_d) \dd \gamma(x_1,\dotsc,x_d)\]
		where $ \Gamma_{m_0}(\mu_1,\dotsc,\mu_d)$ denotes the set of   
		non-negative measures $\gamma$ on $X_1\times\dotsm\times X_d$ with
		total mass $m_0$, such that its $i-$marginal $\gamma_i$ is dominated by $\mu_i$ for each $i$; that is $\int_{X_i}\phi\dd\gamma_i \leq \int_{X_i}\phi\dd\mu_i$ for all non-negative test functions $\phi\in\Class^0(X_i)$.
		
		This is known as the  \emph{optimal partial transport problem} when $d=2$ \cite{CaffarelliMcCann10, Figalli10}, and the \emph{multi-marginal optimal partial transport problem} \cite{KitPass} when $d>2$.  
		
		As we will show below, the problem is in fact equivalent to an ordinary multi-marginal optimal transport problem with an additional marginal; see Theorem \ref{thm: equivalence with ot} and Remark \ref{rem: partial transport is multi-marginal transport} below.
	\end{remark}

	\begin{proposition}
		Let $b:X_1 \times X_2 \times \dotsm\times X_d \rightarrow \mathbb{R}$ be an upper semi-continuous function bounded from below.  Then 
		the map $\gamma\mapsto
		\mathcal{R}_\alpha(b_\#\gamma)$ is concave. 
	\end{proposition}
	
	\begin{proof}
		Using \eqref{eq:R-alpha-variational} one gets
		\begin{align*}
			\mathcal{R}_\alpha(b_\#\gamma)
			&= \max_{\sigma \in \Gamma(\alpha_\sharp\Leb_{[0,1]}, b_\#\gamma)} \int  x_0y \dd \sigma(x_0,y) \\
			&= \max_{\pi \in \Gamma(\alpha_\sharp\Leb_{[0,1]}, \gamma)} \int x_0 b(z) \dd \pi(x_0,z),
		\end{align*}
		where $z=(x_1,\dotsc,x_d)$.
		Then concavity of $\gamma\mapsto
		\mathcal{R}_\alpha(b_\#\gamma)$ on $\Prob(\Rsp^d)$ follows as in Corollary \ref{coro: concavity}.
	\end{proof}
	
	


	


\section{Equivalence between maximizing spectral risk measures and multi-marginal optimal transport}\label{section: equivalence}
Our first main contribution is to show that the spectral risk maximization problem \eqref{eqn: max alpha risk} is equivalent to the multi-marginal optimal transport problem  \eqref{eqn: multi-marginal ot} with $X_0 = [\alpha(0),\alpha(1)] \subseteq \mathbb{R}$, $\mu_0 = \alpha_\sharp\Leb_{[0,1]}$, the other $X_i$ and $\mu_i$ representing the domains and distributions of the underlying variables, respectively, and

\begin{equation}\label{eqn: effective surplus function}
	s(x_0,x_1,\dotsc,x_d) =x_0b(x_1,\dotsc,x_d).
\end{equation}

\begin{lemma}\label{lem: monotone 2d solutions}
	Suppose that $\pi \in \Gamma(\mu_0,\mu_1,\dotsc,\mu_d)$  and let $\gamma$ be the $(1,...,d)$ -marginal of $\pi$. 
	Then, for $s$ given by \eqref{eqn: effective surplus function},
	$$
	\int_{\times_{i=0}^d X_i} s(x_0,x_1,\dotsc,x_d)\dd\pi(x_0,x_1,\dotsc,x_d) \leq  \mathcal{R}_\alpha(b_\#\gamma)
	$$
	
	Furthermore, we have equality if and only if the support of \[ \tau_\pi =\Big((x_0,x_1,x_2,\dotsc,x_d)\mapsto (x_0,b(x_1,x_2,\dotsc,x_d))\Big)_\#\pi \in \mathcal{P}(\mathbb{R}^2) \] is monotone increasing.
\end{lemma}

\begin{proof}
	We first note that
	\begin{equation*}
	\begin{split}
		\int_{\times_{i=0}^d X_i} s(x_0,x_1,\dotsc,x_d)\dd \pi(x_0,\dotsc,x_d) &=\int_{\times_{i=0}^d X_i}x_0b(x_1,\dotsc,x_d)\dd\pi(x_0,\dotsc,x_d)\\
		&=\int_{\mathbb{R}^2}zx_0\dd \tau_\pi(x_0,z)\\
		& \leq  \mathcal{R}_\alpha(b_\#(\gamma)),
		\end{split}
	\end{equation*}
	by Lemma \ref{lem: ot formula for R}.  Furthermore, the uniqueness in Theorem \ref{thm:monge state} implies that the inequality is in fact an equality if and only if the coupling $\tau_\pi$ has monotone increasing support.
\end{proof}
\begin{lemma}\label{lem: glueing}
Given any measure $\gamma \in \Gamma(\mu_1,\dotsc,\mu_d)$, there exists a $\pi \in \Gamma(\mu_0,\mu_1,\dotsc,\mu_d)$ whose $(1,\dotsc,d)$-marginal is $\gamma$, such that $\Big((x_0,x_1,\dotsc,x_d) \mapsto (x_0,b(x_1,x_2,\dotsc,x_d)\Big)_\#\pi$ has monotone increasing support.
\end{lemma}
\begin{proof}
Let $\tau =(F_{\mu_0}^{-1},F_{b_\#\gamma}^{-1})_\# \Leb_{[0,1]}$ be the comonotonic coupling of $\mu_0$ and $b_\#\gamma$.  We disintegrate $\tau(x_0,y) =\tau^y(x_0)\otimes (b_\#\gamma)(y)$ with respect to its second marginal, so that for any measurable function $g( 
x_0,y)$ we have 
$$
\int_{\mathbb{R}^2} g(x_0,y)\dd\tau(x_0,y) = \int_{\mathbb{R}}\int_{\mathbb{R}}g(x_0,y)\dd\tau^y(x_0)\dd(b_\#\gamma)(y),
$$
and $\gamma(x_1,...,x_d) =\gamma^y(x_1,...,x_d) \otimes (b_\#\gamma)(y)  $ with respect to $b$, so that for any measurable $h(x_1,...,x_d)$ we have
$$
\int_{X_1 \times\dotsm\times X_d}h(x_1,\dotsc,x_d)\dd\gamma(x_1,\dotsc,x_d) =\int_{\mathbb{R}} \int_{b^{-1}(y)}h(x_1,\dotsc,x_d)\dd\gamma ^y(x_1,\dotsc,x_d) \dd(b_\#\gamma).
$$

Any measure $\pi(x_0,x_1...,x_d) =\pi^y(x_0,x_1,...,x_d) \otimes (b_\#\gamma)(y)$ on $X_0 \times (X_1 \times\dotsm \times X_d)$, whose conditional probabilities $\pi^y \in \Gamma(\tau^y,\gamma^y) \subseteq \mathcal{P}(X_0\times X_1\times...\times X_d)$ after disintegrating with respect to the mapping $(x_0,\dotsc,x_d) \mapsto (x_0,b(x_1,...,x_d))$ are couplings of $\tau^y$ and $\gamma^y$ for $b_\#\gamma$ almost every $y$ then satisfies the requirement.
\end{proof}
We are now ready to establish the equivalence between the maximal $\alpha-$risk problem \eqref{eqn: max alpha risk} and the multi-marginal optimal transport problem \eqref{eqn: multi-marginal ot}.
\begin{theorem}\label{thm: equivalence with ot}
A probability measure $\pi$ in $\Gamma(\mu_0,\mu_1,\dotsc,\mu_d)$ is optimal for \eqref{eqn: multi-marginal ot} with cost function \eqref{eqn: effective surplus function} if and only if its $(1,\dotsc,d)$-marginal  
is optimal in \eqref{eqn: max alpha risk}, and $ \tau_\pi=\Big((x_0,x_1,x_2,\dotsc,x_d)\mapsto (x_0,b(x_1,x_2,\dotsc,x_d))\Big)_\#\pi$ has monotone increasing support.
\end{theorem}
\begin{proof}	
First, suppose that $\pi$ is optimal in \eqref{eqn: multi-marginal ot} and let $\gamma$ be its $(1,...,d)$-marginal.  Then for any other $\tilde \gamma \in \Gamma(,\mu_1,...,\mu_d)$, construct $\tilde \pi \in \Gamma(\mu_0,\mu_1,...,\mu_d)$ as in Lemma \ref{lem: glueing}.  We then have, by Lemma \ref{lem: monotone 2d solutions} and optimality of $\pi$
\begin{eqnarray*}
	\mathcal{R}_\alpha(b_\#\tilde \gamma) &=&\int_{X_0 \times X_1\times\dotsm\times X_d}s(x_0,...,x_d)\dd\tilde \pi(x_0,\dotsc,x_d) \\
	&\leq& \int_{X_0 \times X_1\times\dotsm\times X_d}s(x_0,\dotsc,x_d)\dd \pi(x_0,\dotsc,x_d) \\
	&\leq& \mathcal{R}_\alpha(b_\# \gamma).
\end{eqnarray*}
This establishes that $\gamma$ is optimal in \eqref{eqn: max alpha risk}.  Furthermore, we must have equality throughout if we choose $\tilde \gamma=\gamma$.  In particular, this implies that $\int_{X_0 \times X_1\times\dotsm\times X_d}s(x_0,\dotsc,x_d)\dd \pi(x_0,\dotsc,x_d) = \mathcal{R}_\alpha(b_\# \gamma)$, which, by the second assertion in Lemma \ref{lem: monotone 2d solutions}, implies that $\tau_\pi$ has monotone increasing support.

Conversely, assume that $\pi\in \Gamma(\mu_0,\mu_1,\dotsc,\mu_d)$, such that its $(1,\dotsc,d$)-marginal $\gamma$  is optimal in \eqref{eqn: max alpha risk} and $\pi$ couples $\mu_0$ and $b_\#\gamma$ monotonically.  Let now $\tilde \pi$ be an optimizer in \eqref{eqn: multi-marginal ot} and $\tilde \gamma$ its $(1,\dotsc,d)$-marginal.  We then have, using Lemma \ref{lem: monotone 2d solutions}, 
\begin{eqnarray*}
\mathcal{R}_\alpha(b_\#\gamma)&=&\int_{X_0\times X_1 \times \dotsm\times X_d} s(x_0,x_1,\dotsc,x_d)\dd\pi(x_0,x_1,\dotsc,x_d)\\ &\leq& \int_{X_0\times X_1 \times \dotsm\times X_d} s(x_0,x_1,\dotsc,x_d)\dd\tilde \pi(x_0,x_1,\dotsc,x_d) \\
&\leq& \mathcal{R}_\alpha(b_\#\tilde \gamma) 
\end{eqnarray*}	
Optimality of $\gamma$ in \eqref{eqn: max alpha risk} then implies that we must actually have equality; in particular 
$\pi$ is optimal in \eqref{eqn: multi-marginal ot} as desired.
\end{proof}



\begin{remark}\label{rem: partial transport is multi-marginal transport}
Returning to the Expected Shortfall Case, $\alpha = \alpha_{m_0} = \frac{1}{m_0} \one_{[1-m_0,1]}$, in view of Remark \ref{rem:: expected shortfall -partial transport},
Theorem~\ref{thm: equivalence with ot} shows that the optimal partial transport problem for the surplus $b(x_1,\dotsc,x_d)$
can be transformed into a multi-marginal transport problem for the surplus
$s(x_0,\dotsc,x_d)= x_0 b(x_1,\dotsc,x_d)$ and additional marginal $\mu_0=\alpha_\sharp\Leb_{[0,1]}$.  To the best of our knowledge, this equivalence between these two well studied mathematical problems has not been observed before in the optimal transport literature.  

In addition, this perspective, together with results in \cite{CaffarelliMcCann10,Figalli10,KitPass} allows us to immediately identify conditions under which the active (that is, the part that couples to $\alpha>0$) part of the optimal $\gamma$ in \eqref{eqn: max alpha risk} is uniquely determined and concentrates on a graph over $x_1$.  Furthermore, algorithms to compute the solution are readily available \cite{BenamouEtal2015, IgbidaNguyen2018}.

These questions will be revisited for more general spectral risk measures later on.
\end{remark}
We close this section with the following immediate consequence of Theorem \ref{thm: equivalence with ot} and a standard existence result in optimal transport theory.
\begin{corollary}
Suppose that $b$ is bounded above and upper-semicontinuous on $X_1 \times X_2 \times\dotsm\times X_d$.  Then there exists a maximizing $\gamma$ in \eqref{eqn: max alpha risk}.
\end{corollary}
\begin{proof}
Note that the boundedness of $\alpha$ implies that $s$ given by \eqref{eqn: effective surplus function} is bounded above and upper-semicontinuous on  $X_0\times X_1 \times X_2 \times\dotsm\times X_d$.  The existence of an optimal $\pi$ in \eqref{eqn: multi-dimensional problem} is then well known; see for example, Theorem 1.7 in \cite{santambook}.  The $(1,\dotsc,d)-$ marginal of $\pi$ then maximizes \eqref{eqn: max alpha risk} by Theorem \ref{thm: equivalence with ot}.
\end{proof}


\section{Solutions for one-dimensional assets and compatible payouts}
\label{section:1d}
We now turn our attention to the structure of maximizers in \eqref{eqn: max alpha risk} when the underlying variables are one dimensional.

Assume that each $x_i\in \mathbb{R}$ and each $\mu_i$ is supported on an interval, $X_i=[\underline x_i, \overline x_i]$, for all $i=1,\dotsc,d$.  We note that in our setting, the first marginal, $\mu_0$ is supported always on the interval $[\underline x_0, \overline x_0]:=[\alpha(0),\alpha(1)]$ with $\alpha(0) \geq 0$.  
\begin{lemma}
Suppose that $b$ is compatible and  monotone increasing in each $x_i \in S_+$ and monotone decreasing for each $x_i \in S_-$.   Then the $s$-comonotone coupling \eqref{eqn: monotone solution} is optimal for \eqref{eqn: multi-marginal ot} with surplus function given by \eqref{eqn: effective surplus function}.  The maximal value is 
\begin{equation}\label{eqn: optimal value}
\int_0^1\alpha(m)b(G_1(m),G_2(m),\dotsc,G_d(m))\dd m.
\end{equation}
Furthermore, if in addition the monotonicity of $b$ with respect to each argument is strict, the solution is unique on the support of $\alpha$; that is, any other solution $\bar \pi$ 
coincides with $\pi =(G_\alpha,G_1,\dotsc,G_d)_\#\Leb_{[0,1]}$ on $(0,\alpha(1)] \times  [\underline x_1, \overline x_1] \times \dotsm\times[\underline x_d, \overline x_d]$.  
\end{lemma} 
Clearly full uniqueness of the optimal $\pi$ cannot hold if $\alpha=0$ on a set of positive measure, or, equivalently, $\mu_0(\{0\}) >0$; in this case, we can rearrange the part of the $(1,...,d)-$ marginal $\gamma$ of $\pi$ that couples to $\alpha=0$ in any way without affecting the value of $\int s\dd\pi$.   The preceding Lemma identifies conditions under which this is the only source of non-uniqueness, so that the optimal $\pi$ is uniquely determined on $(0,\alpha(1)] \times  [\underline x_1, \overline x_1] \times \dotsm\times[\underline x_d, \overline x_d]$.  

Note also that we are able to obtain uniqueness off the set where $\alpha =0$ despite the fact that the surplus function may not be strongly compatible.  This is done by exploiting an idea developed in \cite{pass2022monge} (see in particular Proposition 4.1  there); namely that one only needs strongly compatible interactions between certain key pairs of variables (rather than all of them).

\begin{proof}
The conditions on $b$ make the function $s$ compatible, and so Proposition \ref{prop: non unique monge state} implies the optimality of the $s$-comonotone $\pi$. 
Note that since $\alpha$ is in fact the quantile function for the first marginal $\mu_0 = \alpha_\sharp\Leb_{[0,1]}$, \eqref{eqn: optimal value} is exactly $\int s \dd\pi$.

Now, to prove the last assertion, we will take advantage of the fact that compatibility is equivalent to supermodularity up to a change of variables (recall Remark \ref{rem: supermodularity vs compatibility}).  After such a change of variables, our assumptions become that $b$ is supermodular, and strictly increasing in each argument, and each $G_i =F^{-1}_{\mu_i}$ is the quantile map of  $\mu_i$.

It will suffice to prove the following claim: for any optimizer $\bar \pi$, if $(x_0,x_1,\dotsc,x_d)$ and $(\tilde x_0,\tilde x_1,\dotsc,\tilde x_d)$ are both in the support of $\bar \pi$ and $x_0>\tilde x_0\geq 0$, then  $x_i \geq \tilde x_i$ for each $i$.  To see that this is sufficient, note that applying the claim for $\tilde x_0=0$ implies that the lowest $a$ portion of each mass, $\{x_i <G_i(a)\}$ must pair with $x_0=0$, where $a = \mu_0(\{0\})$. Applying the claim with $\tilde x_0 >0$ then  implies that the support of $\bar\pi$ on $(0,\alpha(1)] \times  [\underline x_1, \overline x_1] \times\dotsm\times[\underline x_d, \overline x_d]$ must be monotone increasing, which immediately implies the desired result.

To see the claim, we apply the $s$-monotonicity property found, in, for example \cite{Pass2012,carlier2003class}.  We define  $x^+_i=\max\{x_i,\tilde x_i\} $ and $x^-_i=\min\{x_i,\tilde x_i\} $.  We then have
\begin{equation}\label{eqn: s-monotonicity}
s(x^+_0,x^+_1,\dotsc,x^+_d)+	s(x^-_0,x^-_1,\dotsc,x^-_d) \leq s(x_0,x_1,\dotsc,x_d)+	s(\tilde x_0,\tilde x_1,\dotsc,\tilde x_d). 
\end{equation}
Note that $x_0^-=\tilde x_0$.  If $\tilde x_0=0$, then $s(x^-_0,x^-_1,\dotsc,x^-_d) =s(\tilde x_0,\tilde x_1,\dotsc,\tilde x_d)=0$, so that \eqref{eqn: s-monotonicity} becomes (after dividing by $x_0$ )
$$
b(x^+_1,\dotsc,x^+_d) \leq b(x_1,\dotsc,x_d)
$$
As each $x^+_i \geq x_i$ and $b$ is strictly monotone in each coordinate, this can only happen if each $x_i^+=x_i$, which is equivalent to the claim.  

On the other hand, if $\tilde x_0>0$, we follow the approach in the proof of Proposition 4.1 in \cite{pass2022monge}.  A straightforward calculation yields that
\begin{align}
&s(x^+_0,x^+_1,\dotsc,x^+_d)+s(x^-_0,x^-_1,\dotsc,x^-_d) - s(x_0,x_1,\dotsc,x_d)-s(\tilde x_0,\tilde x_1,\dotsc,\tilde x_d)\nonumber\\
&=\int_0^1\int_0^1\sum_{i \neq j=0}^d \dfrac{\partial^2 s(x(\theta,\phi))}{\partial x_i\partial x_j }(x_i^+-x_i)(x_j-x_j^-)\dd\theta \dd\phi \label{eqn: partial difference}
\end{align}
where \[x(\theta,\phi) = \phi[\theta x^+ +(1-\theta) x] +(1-\phi)[\theta \tilde x +(1-\theta) x^-].\]  Now, by supermodularity of $s$, each 
$ \dfrac{\partial^2 s(x(\theta,\phi))}{\partial x_i\partial x_j }\geq 0$, and as each $x_i^+-x_i \geq 0$ and $x_j-x_j^- \geq 0$, we have that \eqref{eqn: partial difference} is greater than or equal to $0$; by \eqref{eqn: s-monotonicity}, it must then be equal to $0$, which then implies that each individual term $ \dfrac{\partial^2 s(x(\theta,\phi))}{\partial x_i\partial x_j }(x_i^+-x_i)(x_j-x_j^-)$ in the sum must be $0$.  Now, strict monotonicity of $b$ with respect to each $x_i$ implies that when we take $j=0$, we have $\dfrac{\partial^2 s(x(\theta,\phi))}{\partial x_i\partial x_0 } =\dfrac{\partial b(x(\theta,\phi))}{\partial x_i}>0$, which, as  $x_0>\tilde x_0 =x_0^- $ by assumption, implies that each $x_i^+ = x_i$.  This completes the proof of the claim, and, consequently, the Lemma.



\end{proof}
Thanks to Theorem \ref{thm: equivalence with ot}, the preceding result then easily yields the following characterization of solutions to \eqref{eqn: max alpha risk}.  
\begin{theorem}\label{thm: optimizer is alpha risk}
Assume that $b$ is compatible,  monotone increasing in each $x_i \in S_+$ and monotone decreasing in each $x_i \in S_-$.  Then  $(G_1,\dotsc,G_d)_\#\Leb_{[0,1]}$ maximizes \eqref{eqn: max alpha risk} and the maximal value is given by \eqref{eqn: optimal value}.  If in addition the monotonicity is strict, then letting $a =\mu_0\{0\}$ any other maximizer $\gamma$ must couple the regions defined by	$G_i([0,a])$, and must coincide with the $b$-comonotone coupling $(G_1,\dotsc,G_d)_\#\Leb_{[0,1]}$ on $G_1([a,1]) \times G_2([a,1]) \times\dotsm\times G_d([a,1])$.
\end{theorem}
Note that in the special case when $b$ is supermodular and increasing with respect to each argument, this asserts optimality of the comonotone coupling
Note that each  region $ G_i([0,a])$ is of mass $a$; it is an interval of the form $[\underline{x_i},G_i(a)]$ if $i \in S_+$ (corresponding to the lowest fraction of the mass) and of the form $[G_i(a),\overline{x_i}]$ if $i \in S_-$ (corresponding to the highest fraction of the mass).

\begin{example} Consider $b(x_1,\dotsc,x_d) = x_1 + \dotsb + x_d$.  This corresponds to the value of a portfolio composed of underlying assets $x_1,\dotsc,x_d$.  It is clearly weakly supermodular and therefore compatible, and strictly monotone increasing in each coordinate, so the $(1,...,d)$-marginal of the $\pi$ defined by \eqref{eqn: monotone solution} is  optimal by Theorem \ref{thm: optimizer is alpha risk}.

Explicitly, the optimal value in \eqref{eqn: max alpha risk} is:
$$ 
\mathcal{R}_\alpha(b_\#\gamma) = \int_0^1 \alpha(m)(F_{\mu_1}^{-1}(m)+ \dotsb+F_{\mu_d}^{-1}(m)) \dd m =\sum_{i=1}^d \mathcal{R}_\alpha(\mu_i), $$
which corresponds to coupling ``good'' and ``bad'' events (\ie with $x_i$ small or large, respectively) together.    

Note that this special case could alternatively be deduced by noting that spectral risk measures are sub-additive for all couplings, and additive for the comonotone coupling; see, for instance, \cite{Puccetti2013} for the expected shortfall case (arguments for other $\alpha$ are similar).
\end{example}

\begin{example} If $b(x_1,\dotsc,x_d) =- \sum_{1\leq i< j \leq d} (x_i - x_j)^2$ and if $\alpha \equiv 1$, problem \eqref{eqn: max alpha risk} is equivalent to 
the computation of Wasserstein barycenters \cite{Carlier_wasserstein_barycenter}, while with $\alpha
= \alpha_{m_0}$ we get the partial Wasserstein barycenter problem \cite{KitPass}. In
both cases, since $b$ is supermodular 
the answer can be calculated explicitly by Theorem \ref{thm: optimizer is alpha risk}.
\end{example}
\begin{example}
For $\alpha =1$, it is well known that optimal transport for the surplus function in the preceding example is equivalent to the same problem with $b(x_1,\dotsc,x_d) =|\sum_{i=1}^dx_i|^2$ (see, for example, \cite{gangbo1998optimal}).  Since the mean $M:=\int_{X_1 \times\dotsm\times  X_d}\sum_{i=1}^dx_i \dd\gamma(x_1,\dotsc,x_d) = \sum_{i=1}^d\int_{X_i}x_id\mu_i(x_i)$ of the sum $z=\sum_{i=1}^dx_i$ is completely determined by the marginals $\mu_i$ for any $\gamma \in \Gamma(\mu_1,\dotsc,\mu_d)$, \eqref{eqn: multi-marginal ot} is equivalent to maximizing the variance of the sum $z$ over $\gamma \in \Gamma(\mu_1,\dotsc,\mu_d)$.   Interpreting this sum as the losses of a portfolio comprised of assets with losses $x_i$, we note the variance is sometimes used in this setting as a measure of risk. One drawback to this approach is its symmetry; variance punishes  good outcomes (when the sum $\sum_{i=1}^dx_i$ is below the mean) as well as bad ones (when the sum $\sum_{i=1}^dx_i$ is above the mean).

Choosing $\alpha
= \alpha_{m_0}$ in \eqref{eqn: alpha risk} on the other hand measures only downside risk beyond a particular level.  Since the quadratic surplus is clearly supermodular, Theorem \ref{thm: optimizer is alpha risk} implies that the worse case scenario can be computed explicitly.

More generally, surplus functions of the form $b(x_1,\dotsc,x_d) =H(\sum_{i=1}^dx_i)$  for a convex function $H$ are supermodular so Theorem \ref{thm: optimizer is alpha risk} applies in this setting as well.  Note that the case $H(z) =\max(z-K,0)$ reflects the payout of a basket call option on the underlying assets.
\end{example}

Our last example requires more extensive explanation and so we devote a separate subsection to it.
\subsection{Sensitivity analysis and maximal river flow}\label{subsect: river example}
In this section we briefly describe a recurring example from \cite{iooss2015review}, used throughout that paper to illustrate issues in sensitivity analysis.  In that setting, one wants to understand the influence of the dependence structure (among other factors) between several contributing inputs on an output behavior.

We consider a simple model which involves the height of river at risk of flooding and compares it to the height of a dyke which protects industrial facilities.  The maximal annual overflow $S$ of a river is modelled by
\[
S=Z_\nu+\Bigg(\dfrac{Q}{BK_s\sqrt{\frac{Z_m-Z_\nu}{L}}}\Bigg)^{0.6}-H_d-C_b,
\]
where $\Bigg(\dfrac{Q}{BK_s\sqrt{\frac{Z_m-Z_\nu}{L}}}\Bigg)^{0.6}$ is the maximal annual height of the river and the variables $Q,K_s,Z_\nu,Z_m,H_d,C_b,L,B$ are  physical quantities whose values are modelled as random variables due to their variation in time and space, measurement inaccuracies, or uncertainty of their true values.  Their individual distributions are modelled in \cite{iooss2015review} (see Table 1 on p.4), and are all absolutely continuous with respect to Lebesgue measure.  The  $\alpha$-risk \eqref{eqn: alpha risk}, where the $x_i$ are the variables $Q,K_s,Z_\nu,Z_m,H_d,C_b,L,B$ and $b=S$ is the overflow, quantifies the risk of overflow.  In \cite{iooss2015review}, the variables were assumed to be independent, although other dependence structures are certainly possible; \eqref{eqn: max alpha risk} asks what is the maximal risk over all possible dependence structures. Notice then that the surplus function $b$ is compatible and satisfies the strict monotonicity with respect to each variable required in Theorem \ref{thm: optimizer is alpha risk}.  The $(1,...,d)$ marginal $\gamma$ of the $s$-comonotone solution $\pi$ defined in \eqref{eqn: monotone solution} is  optimal, and the unique optimizer on the support of $\alpha$, $\{\alpha >0\}$.  In this case, the $G_i$ corresponding to $Z_\nu, Q$ and $L$ are monotone increasing, while the $G_i$ corresponding to the other variables are decreasing.




They can be retrieved explicitly by computing the quantile functions associated to each marginals as described above.

\section{Higher dimensional assets}\label{section:hdim}
There is an analogous theory of multi-marginal optimal transport when the underlying variables lie in more general spaces.  Although in these cases it is generally not possible to derive explicit solutions as we did above, it is possible to prove analogous structural properties of optimal couplings, namely, that solutions are of Monge type (that is, concentrated on graphs over $x_1$), for certain surplus functions $s$ (see \cite{gangbo1998optimal} for an early result in this direction, for a particular $s$, and \cite{Kim&PassMongeSol} and \cite{Pass11} for general, sufficient conditions on $s$).

We present one such result below, when each $X_i \subseteq \mathbb{R}^n$, to illustrate how the theory can be adapted to the present setting.  For simplicity, we restrict our attention to the $d=2$ case.  The conditions we impose on $b$ are closely related to the local  differential condition found in \cite{Pass11}; however, the proof in \cite{Pass11} was restricted to the case when each marginal was supported on a space of the same dimension, whereas here our first marginal $\mu_0$, corresponding to $x_0=\alpha$, has one dimensional support while the other marginals are supported in $\mathbb{R}^n$.  We modify that proof to fit the present setting here. Similar results can be established for larger $d$, using similar arguments.  However, the conditions imposed on $b$ become more stringent for larger $d$, and more complicated to state.  

In what follows, we will assume that $b$ is $\mathcal C^2$; $D_{x_i}b(x_1,x_2) =(\frac{\partial b}{\partial x_i^1},\frac{\partial b}{\partial x_i^2},...,\frac{\partial b }{\partial x_i^n})$ represents the gradient of the function $b:\mathbb{R}^d \times \mathbb{R}^d \rightarrow \mathbb{R}$ with respect to the variable $x_i \in \mathbb{R}^n$, for $i=1,2$.  Similarly, $D^2_{x_1x_2}b(x_1,x_2) =\big(\frac{\partial^2 b}{\partial x_1^r \partial x_2^l}\big)_{rl}$ is the $n$ by $n$ matrix of mixed second order derivatives (each entry is a derivative with respect to one coordinate from each of $x_1$ and $x_2$).

\begin{theorem}\label{prop: higher dimensional assets monge sol}
Suppose that $d=2$, that the domains $X_1,X_2 \subseteq \mathbb{R}^n$ are compact and that $\mu_1$ is absolutely continuous with respect to Lebesgue measure.  Assume that $x_2 \mapsto D_{x_1}b(x_1,x_2)$ is injective for each fixed $x_1 \in X_1$, and that for each $(x_1,x_2) \in X_1 \times X_2$, $\det(D^2_{x_1x_2}b(x_1,x_2)) \neq 0$ and
\begin{equation}\label{eqn: multi-marginal differential condition}
D_{x_2}b(x_1,x_2)\cdot[D^2_{x_1x_2}b(x_1,x_2)]^{-1}D_{x_1}b(x_1,x_2) >0.
\end{equation}

Then the part of the solution to \eqref{eqn: max alpha risk} away from $\alpha =0$ concentrates on the graph of a function over $x_1$.  Furthermore, if  $|\{\alpha =0\}| =0$, the solution is unique.
\end{theorem}
\begin{remark}
The injectivity of $x_2 \mapsto D_{x_1}b(x_1,x_2)$ is known as the \emph{twist condition} in the optimal transport literature; it is well known that, together with the absolute continuity of $\mu_1$, it guarantees the Monge structure of the solution to the two marginal optimal transport problem with surplus $b$ (see, for example, \cite{santambook}).  The invertibility of $D^2_{x_1x_2}b(x_1,x_2)$ is frequently referred to as \emph{non-degeneracy}, and can be seen as a linearized version of the twist. Here we require additional hypotheses as we are dealing with the more sophisticated 3 marginal problem with surplus $s(\alpha, x_1,x_2) =\alpha b(x_1,x_2)$, which is equivalent to the spectral risk maximization \eqref{eqn: max alpha risk} by Theorem \ref{thm: equivalence with ot}.
\end{remark}
\begin{proof}  

By Theorem \ref{thm: equivalence with ot}, we can consider the multi-marginal optimal transport problem \eqref{eqn: multi-marginal ot} rather than \eqref{eqn: max alpha risk}.  Let $\pi$ solve  \eqref{eqn: multi-marginal ot} and $u_0,u_1,u_2$ the optimal solutions to the dual problem \eqref{eqn:dualit-multi-ot}.  Then $u_1$ is differentiable $\mu_1$ almost everywhere, by a standard argument, originally found in \cite{McCann2001}.  Fix an $x_1$ where $u_1$ is differentiable.  To show that the solution concentrates on a graph, we need to prove that there is only one $x_2$ and one non-zero $\alpha$ such that $(\alpha,x_1,x_2)$ is in the support of $\pi$.

At such points, the inequality constraint in \eqref{eqn:dualit-multi-ot}, together with the equality \eqref{eqn: opt condition} on the support of the optimizer and the envelope theorem implies, whenever $(\alpha,x_1,x_2)$ is in the support of the optimal $\pi$
$$
Du_1(x_1) =\alpha D_{x_1}b(x_1,x_2).
$$
Note that for fixed $x_1$, the twist condition implies that for $\alpha \neq 0$, this uniquely defines $x_2 =x_2(\alpha)$ as a function of $\alpha$ (that is, for each $\alpha$, there is at most one $x_2$ satisfying this equation).  Furthermore, we can differentiate implicitly with respect to $\alpha$ to obtain, 
$$
0=	D_{x_1}b(x_1,x_2(\alpha)) +D^2_{x_1x_2}b(x_1,x_2(\alpha)) D_\alpha x_2(\alpha), 
$$
or
\begin{equation}\label{eq:Dx2}
D_\alpha x_2(\alpha) = -[D^2_{x_1x_2}b(x_1,x_2(\alpha)) ]^{-1}D_{x_1}b(x_1,x_2(\alpha)).
\end{equation}
Now, it is well known (see for example \cite{gangbo1998optimal} or \cite{Pass11}) that the $u_i$ can be taken to be $s$-conjugate.  Explicitly, this means that $u_0$, the potential corresponding to $x_0=\alpha$, can be written as
$$
u_0(\alpha) = \sup_{x_1,x_2}[\alpha b(x_1,x_2) - u_1(x_1) -u_2(x_2)],
$$
and is therefore convex, as a supremum of affine functions.  If   $(\alpha,x_1,x_2)$ is in the support of the optimal $\pi$, by the optimality condition presented in section \ref{section:prelim}, we must have that $x_1,x_2$ attains the supremum above, so that $b(x_1,x_2)$ lies in the subdifferential of $u_0$ at $\alpha$:
\[
b(x_1,x_2) \in \partial u_0(\alpha).
\]
But we also have $x_2 =x_2(\alpha)$, so that 
\begin{equation}\label{eq:inclusion}
b(x_1,x_2(\alpha)) \in \partial u_0(\alpha).
\end{equation}
We claim that this can hold for at most one value of $\alpha$.  Indeed, differentiating $b(x_1,x_2(\alpha))$ with respect to $\alpha$  and using \eqref{eq:Dx2} yields
\[
\begin{split}
 D_\alpha b(x_1,x_2(\alpha))&=D_{x_2}b(x_1,x_2(\alpha))\cdot D_\alpha x_2(\alpha)\\
 &= -D_{x_2}b(x_1,x_2(\alpha))\cdot [D^2_{x_1x_2}b(x_1,x_2(\alpha)) ]^{-1}D_{x_1}b(x_1,x_2(\alpha)).   
\end{split}
\]
Under the given assumption, this quantity is negative, and therefore the left hand side in \eqref{eq:inclusion} is a strictly decreasing function of $\alpha$; it can therefore intersect the subdifferential $\partial u_0(\alpha)$, an increasing set valued function of $\alpha$, at most once.  The points $(\alpha, x_2(\alpha))$, where $\alpha$ is the point at which this intersection occurs, are then the only points such that $(\alpha,x_1,x_2)$ are in the support of $\pi$, which establishes that the solution $\gamma$ to \eqref{eqn: max alpha risk} concentrates on a graph.

Uniqueness then follows by a very standard argument; if $\pi_0$ and $\pi_1$ are both solutions, by linearity of the functional $\gamma_{1/2} =\frac{1}{2}[\gamma_0+\gamma_1]$ is also a solution.  The argument above implies that the supports of $\gamma_0$ and $\gamma_1$  concentrate on graphs $T_0$ and $T_1$ over $x_1$.  The  support of $\gamma_{1/2}$  then concentrates on the union of the graphs of $T_0$ and $T_1$.  However, since $\gamma_{1/2}$ is also a solution, the argument above implies that this set should concentrate on a single graph.  This is not possible unless $T_0 =T_1$, $\mu_1$ almost everywhere, in which case $\gamma_0$ and $\gamma_1$ coincide, establishing the uniqueness assertion.
\end{proof}
There are of course output functions $b$ for which the hypotheses in the preceding result do not hold, and in these cases solutions are often not of Monge type.   Alternatively, one can estimate the dimension of the support of the optimal $\pi$ in \eqref{eqn: multi-marginal ot} using the signature of the off diagonal part of the Hessian of $s(\alpha, x_1,...,x_d)=\alpha b(x_1,...,x_d)$, as in \cite{Passthesis} and \cite{Pass12}.  We illustrate this below with the following result in the $d=2$ case,  which is an immediate consequence of Theorem 3.1.3 and Lemma 3.3.2 in \cite{Passthesis}.
\begin{proposition}
Assume $X_1, X_2 \subseteq \mathbb{R}^n$ and choose any point $(\alpha, x_1,x_2) \in (0,\infty) \times X_1 \times X_2 $ where $D^2_{x_1x_2}b(x_1,x_2)$ is non-singular.  Let $\pi$ be optimal in \eqref{eqn: multi-marginal ot}.  Then in a neighbourhood of   $(\alpha, x_1,x_2)$ the support of $\pi$ is contained in a Lipschitz submanifold of dimension at most $n+1$.  Furthermore, if \eqref{eqn: multi-marginal differential condition} holds, the dimension is at most $n$.
\end{proposition}
For larger $d$, if each $X_i \subseteq \mathbb{R}^{n_i}$ lies in a space of dimension $n_i$, the results in \cite{Passthesis} and \cite{Pass12} imply that the support of $\pi$ will again be contained in a Lipschitz submanifold of dimension $m$, where $m$ depends on the mixed second derivatives of $b$ in an intricate way.  Generically, if each $n_i=n$ is the same, $m$ will lie between $n$ and $(d-1)n+1$.  Of course, the dimension of the support of the solution $\gamma$ to \eqref{eqn: max alpha risk} will be no larger than $m$, as by Theorem \ref{thm: equivalence with ot} $\gamma$ is the $(1,...,d)$-marginal of $\pi$.



\section{Stability}\label{section:stability}
Among other results, in \cite{ghossoub2020maximum} the authors show stability of the optimal value $\mathcal{R}_\alpha(b_\#\gamma)$ and measure $\gamma$ with respect to weak convergence of the marginals and $L^1$ convergence of the payoff function $b$.  Below, we show that similar results (under slightly different assumptions -- see Remark \ref{rem: stability comparison}) can be easily deduced by combining Theorem \ref{thm: equivalence with ot} with optimal transport theory.

In addition, we further leverage the connection to optimal transport to improve  the stability with respect to the marginals when the underlying assets are one dimensional and $b$ is compatible; in this case, 
we show that under appropriate conditions, $\mathcal{R}_\alpha$ in fact exhibits Lipschitz dependence on the marginals.

\begin{proposition}
Suppose that for each $i$ $\{\mu_i^k\}$ is a sequence of probability measures in $\mathcal{P}(X_i)$ converging weakly to $\mu_i$ as $k \rightarrow \infty$ and $\{b^k\}$ is a sequence of continuous output functions on $X_1 \times X_2 \times\dotsm\times X_d$ converging uniformly to a continuous $b$, such that $b$ is bounded above, that is $b \leq C$, and bounded below by a sum of integrable functions, that is, $b(x_1,...,x_d) \geq \sum_{i=1}^du_i(x_i)$ with each $u_i \in L^1(\mu_i)$.  Let $\alpha^k$ be a sequence of spectral functions, uniformly bounded $\alpha^k \leq M$, such that the associated spectral measure $\alpha^k_\#\Leb_{[0,1]}$ converges weakly to $\alpha_\#\Leb_{[0,1]}$, where $\alpha$ is bounded.  
Then
\[\max_{\gamma \in \Gamma(\mu^k_1, \mu^k_2,\dotsc,\mu^k_d)} \mathcal{R}_{\alpha^k}(b^k_\#\gamma) \rightarrow \max_{\gamma \in \Gamma(\mu_1, \mu_2,\dotsc,\mu_d)} \mathcal{R}_{\alpha}(b_\#\gamma).\]  
Furthermore, if $\gamma^k$ maximizes $\mathcal{R}_{\alpha^k}(b^k_\#\gamma)$ over $\Gamma(\mu^k_1, \mu^k_2,\dotsc,\mu^k_d)$, then the weak limit $\gamma$ of any weakly convergent subsequence of the $\gamma^k$ maximizes $\mathcal{R}_{\alpha}(b_\#\gamma)$ over $\Gamma(\mu_1, \mu_2,\dotsc,\mu_d)$.
\end{proposition}
\begin{remark}\label{rem: stability comparison}
This result should be compared with Proposition 7 in \cite{ghossoub2020maximum}; in that setting, stronger assumptions  (uniform Holder continuity) on the $b^k$ are required, although a weaker notion of convergence is imposed on them. In addition, we also obtain stability with respect to variations in the spectral function $\alpha$. 
\end{remark}
\begin{proof}
The proof follows almost immediately from Theorem \ref{thm: equivalence with ot} and a corresponding stability result in optimal transport; namely, \cite[Theorem 5.20]{Villani-OptimalTransport-09}.  Note that as we can take $X_0 =[0,M]$ to be bounded, uniform convergence of the $b^k$ implies uniform convergence of the corresponding $s^k$ defined by \eqref{eqn: effective surplus function}.

The only issue to note is that that \cite[Theorem 5.20]{Villani-OptimalTransport-09} applies only to the two marginal problem.  However, exactly the same argument as is used there applies to the multi-marginal case to show that the multi-marginal version of the $s^k$ -cyclic monotonicity property satisfied by the support of each $\pi^k$ implies that the support of the limit measure $\pi$ is $s$-cyclically monotone.  This then implies that $\pi$ is optimal in \eqref{eqn: multi-marginal ot} by \cite[Theorem 12]{Griessler18}.

\end{proof}

We now specialize to the case when the $x_i$ are one dimensional.  Our main result in this direction requires  the following lemma.  It is in fact a special case of \cite[Corollary 11]{Pichler}; we provide the short proof for the convenience of the reader.

\begin{lemma} \label{lem:R-lipschitz} If $\alpha \leq M$, then $\mu\in \Prob(\Rsp)\mapsto \mathcal{R}_\alpha(\mu)$ is $M$-Lipschitz for the $p$-Wasserstein distance,
\begin{equation}
\abs{\mathcal{R}_\alpha(\mu) - \mathcal{R}_\alpha(\nu)} \leq M \Wass_p(\mu,\nu).
\end{equation}
\end{lemma}
\begin{proof} On $\Rsp$, the $p$-Wasserstein distance is given  by the $\LL^p$ distance between the inverse cumulative distribution functions
$ \Wass_p(\mu,\nu) = \nr{F_\mu^{-1} - F_\nu^{-1}}_{\LL^p([0,1])}.$ The
statement follows:
\begin{align*}
\mathcal{R}_\alpha(\mu) - \mathcal{R}_\alpha(\nu)
&= \int (F_\mu^{-1}(m) - F_\nu^{-1}(m))\alpha(m) \dd m  \\
&\leq \nr{F_\mu^{-1} - F_\nu^{-1}}_{\LL^p([0,1])} \nr{\alpha}_{\infty}.
\end{align*}
\end{proof}
\begin{proposition}\label{prop:continiuty} 
Assume that each $X_i \subset \mathbb{R}$ and that  the output function $b$ is weakly compatible, monotone increasing in each $x_i \in S_+$ and monotone decreasing in each $x_i \in S_-$.  If in addition $b$ is $K$-Lipschitz with respect
to $\nr{\cdot}_p$ on $\Rsp^d$ and $\alpha$ is  bounded by $M$, then:
\begin{equation}
\abs{\max_{\gamma\in \Gamma(\mu_1,\dotsc,\mu_d)} \mathcal{R}_\alpha(b_{\#}\gamma) - \max_{\gamma\in \Gamma(\nu_1,\dotsc,\nu_d)} \mathcal{R}_\alpha(b_{\#}\gamma)} \leq M K\left(\sum_i \Wass_p(\mu_i,\nu_i)\right)^{1/p}
\end{equation}
\end{proposition}
\begin{proof}
By Theorem \ref{thm: optimizer is alpha risk}, the $(1,\dotsc,d) - $ marginals $\gamma =(G_1,\dotsc,G_d)_\#\Leb_{[0,1]}$ and $\tilde \gamma=(\tilde G_1,\dotsc,\tilde G_d)_\#\Leb_{[0,1]}$ of the $s$-monotone couplings $\pi$ and $\tilde \pi$ of the $\mu_i$ and $\nu_i$, respectively, achieve the suprema of $\mathcal{R}_\alpha$.  Therefore, we get that $(G_1,\dotsc,G_d, \tilde G_1,\dotsc,\tilde G_d)_\#\Leb_{[0,1]} \in \Gamma(\gamma,\tilde \gamma) \subset \mathcal{P}(\mathbb{R}^{2d})$, and its projection onto each $X_i \times X_i$  is $(G_i,\tilde G_i)_\#\Leb_{[0,1]}$, which is the optimal coupling between $\mu_i$ and $\nu_i$ for the cost function $|x_i -\tilde x_i|^p$ in the definition of the $p-$Wasserstein distance.  Therefore, setting $G=(G_1,\dotsc,G_d)$ and $\tilde G=(\tilde G_1,\dotsc,\tilde G_d)$, 
\begin{align*}
\Wass_p^p(b_\# \gamma, b_{\#}\tilde\gamma)
&\leq \int_0^1 (b(G(t)) - b(\tilde G(t)))^p \dd t\\
&\leq K^p \int \nr{G(t) - \tilde G(t)}_p^p \dd t\\
&= K^p \int \sum_i |G_i(t) - \tilde G_i(t)|^p \dd t \\
&= K^p \sum_i \Wass_p(\mu_i,\nu_i)^p 
\end{align*}
One concludes using Lemma~\ref{lem:R-lipschitz} on the Lipschitz continuity of $\mathcal{R}_\alpha$. 
\end{proof}


\section{Multidimensional measures of risk}\label{section:multidim}
We consider the framework in \cite{Ruschendorf2006} in which risk is measured in a multi-dimensional way.  Instead of a single output variable, we now have several, depending on the same underlying random variables; this corresponds to a vector valued output function $b:X_1\times\dotsm\times X_d\to\Rsp^n$.  A natural form of multi-variate risk measures on the distribution $b_\#\gamma$ of output variables is then the maximal correlation measure from \cite{Ruschendorf2006}, which is defined by 


\[ \mathcal{R}_\nu(b_{\#}\gamma)=\max_{\eta\in\Gamma(b_{\#}\gamma,\nu)} \int_{\mathbb{R}^n \times \mathbb{R}^n} z\cdot y\; \dd\eta \]
for some probability measure $\nu \in \mathcal{P}(\mathbb{R}^n)$.  We note that it was proven in \cite{ekeland2012comonotonic} that any strongly coherent multi-variate risk measure takes this form.
We consider the problem of maximizing $\mathcal{R}_\nu(b_{\#}\gamma)$ over all $\gamma \in\Gamma(\mu_1,\dotsc,\mu_d)$, where the $\mu_i$ as before represent the distributions of the underlying variables. 
Exactly as in Theorem \ref{thm: equivalence with ot}, one can show that this problem is equivalent to the multi-marginal problem
\begin{equation}\label{eqn: multi-dimensional problem}
\max_{\pi\in\Gamma(\nu, \mu_1,\dotsc,\mu_d)} \int b(x_1,\dotsc,x_d)\cdot y\;\dd\pi. 
\end{equation}
This problem is more challenging than the case of a scalar valued $b$; nonetheless, we are able to obtain the following results in particular cases.
\begin{proposition}
Suppose that the underlying variables $x_i$ are one dimensional.  Assume that $\nu$ is concentrated on a smooth curve, that is $\nu=f_\# \Leb_{[0,1]}$ with $f:[0,1]\to\Rsp^d$, where each component of $f$ is positive and monotone increasing,  and each component $b_j$ of $b=(b_1,b_2,\dotsc,b_n)$ is supermodular and monotone increasing in each $x_i$. Then $\Big((m,x_1,\dotsc,x_d) \mapsto (f(m),x_1,\dotsc,x_d)\Big)_\# \pi$ is optimal in \eqref{eqn: multi-dimensional problem},where $\pi=(\id,F_{\mu_1}^{-1},F_{\mu_2}^{-1},\dotsc,F_{\mu_d}^{-1})_\#\Leb_{[0,1]}$ is the comonotone coupling of $\mu_0, \mu_1,\dotsc,\mu_d$, where $\mu_0=\Leb_{[0,1]}$.
\end{proposition}
\begin{proof}
We rewrite the problem \eqref{eqn: multi-dimensional problem} above as  
\[ \max_{\pi\in\Gamma(\nu, \mu_1,\dotsc,\mu_d)} \int b(x_1,\dotsc,x_d)\cdot y\;\dd\pi =\max_{\pi\in\Gamma(\Leb_{[0,1]}, \mu_1,\dotsc,\mu_d)} \int b(x_1,\dotsc,x_d)\cdot f(m)\;\dd\pi.\]
The objective function $b(x_1,\dotsc,x_d)\cdot f(m) =\sum_{j=1}^d b_j(x_1,\dotsc,x_d)f_j(m)$ is now a supermodular function of the one dimensional variables $m,x_1,x_2,...,x_d$, and so Proposition \ref{prop: non unique monge state} implies the desired result.
\end{proof}
For more general, diffuse $\nu$, we are able to prove that the solution is of Monge form and unique, provided that $b$ is invertible.
\begin{proposition}
Assume that $\nu$ is absolutely continuous with respect to $n$ dimensional Lebesgue measure and  $b:X_1\times\dotsm\times X_d\to\Rsp^n$ is invertible. Then there exists a unique solution to \eqref{eqn: multi-dimensional problem}. Furthermore, it concentrates on a graph over $y$.
\end{proposition}
Note that the invertibility assumption on $b$ implies that the sum of the dimensions of the $X_i$ must be less than or equal to the dimension $n$ of $\nu$
\begin{proof}
Noting that the gradient of the surplus function in \eqref{eqn: multi-dimensional problem} with respect to $y$ is $b$, the invertibility condition on $b$  is the famous twist condition in optimal transport, and so this result follows by a minor refinement of a standard argument which we produce below. 

Letting $u_0(y),u_1(x_1),\dotsc,u_n(x_n)$ be the optimal solutions to the dual problem,  we note that as  $\sum_{i=0}^d u_i(x_i)\geq b(x_1,\dotsc,x_d)\cdot y$, with equality by \eqref{eqn: opt condition} in the support of any optimzer $\pi$, the envelope theorem implies, 
\begin{equation}\label{eqn: first order multi-dimensional}
	Du_0(y) =D_y(y\cdot b(x_1,\dotsc,x_d)) = b(x_1,\dotsc,x_d)
\end{equation}
wherever $u_0$ is differentiable, at points $(y,x_1,\dotsc,x_d)$ in the support of $\pi$.  A now standard argument, found in \cite{McCann2001}, implies that $u_0$ is Lipschitz and therefore differentiable Lebesgue almost everywhere by Rademacher's theorem; the differentiability holds $\nu$ almost everywhere by the absolute continuity of $\nu$.  Therefore, \eqref{eqn: first order multi-dimensional} holds on a set of full $\pi$ measure.  Invertibility of $b$ then means that this equation becomes $(x_1,\dotsc,x_d)=b^{-1}(Du_0(y))$.  Setting $T(y) = b^{-1}(Du_0(y))$, this means that $\pi$ concentrates on the graph of $T:\mathbb{R}^n \rightarrow \mathbb{R}^d$, as desired.

Uniqueness then follows exactly as in the proof of Theorem \ref{prop: higher dimensional assets monge sol}.
\end{proof}

\bibliographystyle{plain}
\bibliography{bibli}
\end{document}